\documentclass[12pt,leqno,twoside]{amsart}
\usepackage{amssymb,amsmath,amsthm,soul,color}
\usepackage{t1enc}
\usepackage[cp1250]{inputenc}
\usepackage{a4,indentfirst,latexsym}
\usepackage{graphics}
\usepackage{mathrsfs}
\usepackage{cite,enumitem,graphicx}
\usepackage[colorlinks=true,urlcolor=blue,
citecolor=red,linkcolor=blue,linktocpage,pdfpagelabels,
bookmarksnumbered,bookmarksopen]{hyperref}
\usepackage[english]{babel}
\usepackage[left=2.61cm,right=2.61cm,top=2.72cm,bottom=2.72cm]{geometry}
\usepackage[colorinlistoftodos]{todonotes}
\usepackage{pifont}
\usepackage{mathrsfs}
\usepackage{cite,enumitem,graphicx}
\input xy
\xyoption{all}

\usepackage{mathrsfs}  

\newtheorem{Th}{Theorem}[section]

\newtheorem{Lem}[Th]{Lemma}

\newtheorem{Rem}[Th]{Remark}

\newenvironment{altproof}[1]
{\noindent
{\em Proof of {#1}}.}
{\nopagebreak\mbox{}\hfill $\Box$\par\addvspace{0.5cm}}

\makeatletter
    
    \newcommand{\Rmnum}[1]{\expandafter\@slowromancap\romannumeral #1@}

   \def\Z{\mathbb{Z}}

   \def\R{\mathbb{R}}

   \def\J{\mathcal{J}}

\newcommand{\cC}{{\mathcal C}}

\newcommand{\cF}{{\mathcal F}}

\newcommand{\cI}{{\mathcal I}}
\newcommand{\cJ}{{\mathcal J}}

\newcommand{\cN}{{\mathcal N}}

\newcommand{\Ga}{\Gamma}

\numberwithin{equation}{section}

\DeclareMathOperator*{\essinf}{ess\,inf}

\begin{document}

\title{Solutions of the fractional Schr\"{o}dinger equation with sign-changing nonlinearity}

\author{Bartosz Bieganowski}
\date{}

\maketitle
\pagestyle{myheadings} \markboth{\underline{B. Bieganowski}}{
\underline{Solutions of the NLFS equation with sign-changing nonlinearity}}

\begin{abstract} We look for solutions to a nonlinear, fractional Schr\"odinger equation
$$(-\Delta)^{\alpha / 2} u + V(x)u = f(x,u)-\Gamma(x)|u|^{q-2}u\hbox{ on }\R^N,$$
where potential $V$ is coercive or $V=V_{per} + V_{loc}$ is a sum of periodic in $x$ potential $V_{per}$ and localized potential $V_{loc}$, $\Gamma\in L^{\infty}(\R^N)$ is periodic in $x$, $\Gamma(x)\geq 0$ for a.e. $x\in\R^N$ and $2<q<2^*_\alpha$. If $f$ has the subcritical growth, but higher than $\Gamma(x)|u|^{q-2}u$, then we find a ground state solution being a minimizer on the Nehari manifold. Moreover we show that if $f$ is odd in $u$ and $V$ is periodic, this equation admits infinitely many solutions, which are pairwise geometrically distinct. Finally, we obtain the existence result in the case of coercive potential $V$.
\end{abstract}

\vspace{0.2cm}


\section*{Introduction}
\setcounter{section}{1}

We consider the following nonlinear, fractional Schr\"odinger equation
\begin{equation}
\label{eq}
(-\Delta)^{\alpha / 2} u + V(x)u = f(x,u)-\Gamma(x)|u|^{q-2}u\hbox{ on }\R^N,\;\alpha \in (0,2],\; N > \alpha,
\end{equation}
with $u\in H^{\alpha / 2}(\R^N)$, which appears in different areas of mathematical physics. Recently, the fractional Schr\"odinger equation has been introduced to describe the propagation dynamics of wave packets in the presence of the harmonic potential and also of the free particle (see \cite{ZhangLiu, ZhangZhong}). In \cite{Longhi} has been proposed an optical realization of this equation, based on transverse light dynamics in aspherical optical cavities. The case of a linear potential is also a fundamental problem in
quantum mechanics that can be treated and solved analytically (see \cite{Guedes, Robinett}). Such an equation was also studied in the quantum scattering problem (see \cite{GuoXu}).

The fractional Laplacian $(-\Delta)^{\alpha / 2}$ of a function $\psi : \R^N \rightarrow \R$ is defined by the Fourier transform by the formula
$$
\cF \left( (-\Delta)^{\alpha / 2} \psi \right) (\xi) := | \xi |^\alpha \hat{\psi} (\xi),
$$
where
$$
\cF \psi (\xi) := \hat{\psi} (\xi) := \frac{1}{(2\pi)^{N/2}} \int_{\R^N} e^{-2\pi i \xi \cdot x} \psi(x) \, dx
$$
denotes the usual Fourier transform. When $\psi : \R^N \rightarrow \R$ is smooth enough, it can be defined by the principal value of the singular integral
$$
(-\Delta)^{\alpha / 2} \psi (x) = c_{N, \alpha} P.V. \int_{\R^N} \frac{\psi(x) - \psi(y)}{|x-y|^{N+\alpha}} \, dy,
$$
where $c_{N,\alpha}$ is some normalization constant. It is known, that $(-\Delta)^{\alpha / 2}$ reduces to $-\Delta$ as $\alpha \to 2^-$ -- see \cite{DiNezza}. In this paper we identify $(-\Delta)^{\alpha / 2}$ with the classical Laplace operator $-\Delta$ for $\alpha = 2$. By the very definition we observe that the fractional Laplacian is non-local (see \cite{Cabre, DiNezza}). 

For $0 < \alpha < 2$, let us remind the definition of the fractional Sobolev space:
$$
H^{\alpha / 2} ( \R^N) := \left\{ u \in L^2 (\R^N) \ : \ \int_{\R^N} |\xi|^\alpha |\hat{u}(\xi)|^2 \,d\xi + \int_{\R^N} |u(x)|^2 \,dx < \infty \right\}.
$$
It is a Hilbert space endowed with the norm
$$
u \mapsto \sqrt{
\int_{\R^N} |\xi|^\alpha |\hat{u}(\xi)|^2 \, d\xi + \int_{\R^N} |u|^2 \, dx
}.
$$

The equation (\ref{eq}) describes the behaviour of so-called standing wave solutions $\Phi(x,t)=u(x)e^{-i\omega t}$ of the following time-dependent fractional Schr\"odinger equation
$$
i \frac{\partial \Phi}{\partial t} = (-\Delta)^{\alpha / 2} \Phi + (V(x)+\omega) \Phi - g(x, |\Phi|).
$$
Such an equation was introduced by Laskin by expanding the Feynman path integral from the Brownian-like to the L\'evy-like quantum mechanical paths (see \cite{Laskin2000, Laskin2002}). The time-dependent equation has been recently studied by A. Liemert and A. Kienle (\cite{LiemertKienle}) with the linear potential $V(x) = \beta x$.

The nonlinear term $f$ satisfies the following conditions:

\begin{itemize}
\item[(F1)] $f:\R^N\times\R\to \R$ is measurable, $\Z^N$-periodic in $x\in\R^N$ and continuous in $u\in\R$ for a.e. $x\in\R^N$, and there are $c>0$ and $2< q<p< 2^*_\alpha := \frac{2N}{N-\alpha}$ such that
$$|f(x,u)|\leq c(1+|u|^{p-1})\hbox{ for all }u \in\R,\; x\in\R^N,$$
\item[(F2)] $f(x,u)=o(|u|)$ uniformly in $x$ as $|u|\to 0^+$,
\item[(F3)] $F(x,u)/|u|^q\to\infty$ uniformly in $x$ as $|u|\to\infty$, where $F(x,u) = \int_0^u f(x,s) \, ds$ is the primitive of $f$ with respect to $u$,
\item[(F4)] $u\mapsto f(x,u)/|u|^q$ is strictly increasing on $(-\infty,0)$ and $(0,\infty)$.
\end{itemize} 

We impose on  $\Gamma$ the following condition: 
\begin{itemize}
\item[($\Gamma$)] $\Gamma\in L^{\infty}(\R^N)$ is $\Z^N$-periodic in $x\in\R^N$, $\Gamma(x)\geq 0$ for a.e. $x\in\R^N$.
\end{itemize} 

Note that the nonlinearity $(x,u) \mapsto f(x,u) - \Gamma(x) |u|^{q-2}u$ does not satisfy the Ambrosetti-Rabinowitz type condition. In fact it may be sign-changing, for example -- consider $f(x,u) = |u|^{p-1}u$ and $\Gamma \equiv 1$, where $2 < q < p < 2^*_\alpha$.

We assume that the potential $V$ satisfies:

\begin{itemize}
\item[($V_\alpha$1)] $V = V_{per} + V_{loc}$, where $V_{per} \in L^{\infty}(\R^N)$ is $\Z^N$-periodic in $x\in\R^N$ and $V_{loc} \in L^{\infty}(\R^N)$ is such that $\lim_{|x|\to\infty} V_{loc}(x) = 0$; moreover 
$$
\left\{ \begin{array}{ll}
V_0 := \essinf_{x\in\mathbb{R}^N} V(x) > 0 & \ \mathrm{for} \ 0<\alpha<2, \\
\inf \sigma (-\Delta + V(x)) > 0 & \ \mathrm{for} \ \alpha=2,
\end{array}
\right.
$$
\end{itemize} 
or
\begin{itemize}
\item[($V_\alpha$2)] $V \in \mathcal{C}(\mathbb{R}^N, \mathbb{R})$ is such that $\lim_{|x| \to \infty} V(x) = \infty$ and 
$$
V_0 := \inf_{x\in\mathbb{R}^N} V(x) > 0.
$$
\end{itemize} 

We consider a Hilbert space
$$
E^{\alpha / 2} := \left\{ u \in L^2 (\R^N) \ : \ \int_{\R^N} |\xi|^\alpha |\hat{u}(\xi)|^2 \, d\xi + \int_{\R^N} V(x) |u|^2 \, dx < \infty \right\}
$$
endowed with the following norm
$$
\|u\|^2 := \int_{\R^N} |\xi|^\alpha |\hat{u}(\xi)|^2 \, d\xi + \int_{\R^N} V(x) |u|^2 \, dx
$$
and the scalar product
$$
\langle u, v \rangle := \int_{\R^N} |\xi|^\alpha \hat{u} (\xi) \overline{\hat{v}(\xi)}\,d\xi + \int_{\R^N} V(x) u(x)v(x)\, dx.
$$

Our goal is to find a {\em ground state} of the energy functional
$\J:E^{\alpha / 2}\to\R$ of class $\cC^1$ given by
$$\J(u):= \left\{  \begin{array}{ll}
\displaystyle  \frac12\int_{\R^N} |\xi|^\alpha |\hat{u}(\xi)|^2\, d\xi +\frac{1}{2} \int_{\R^N} V(x)|u(x)|^2\,dx- \cI(u), \ & \ \mathrm{for} \ 0<\alpha < 2, \\
\displaystyle  \frac12\int_{\R^N} |\nabla u(x)|^2 + V(x)|u(x)|^2\,dx- \cI(u), \ & \ \mathrm{for} \ \alpha = 2,
\end{array}  \right.$$
where
$$
\cI(u) := \int_{\R^N} \Big(F(x,u(x))-\frac1q\Gamma(x)|u(x)|^q\Big)\, dx
$$
i.e. we look for a critical point  being a minimizer of $\J$ on the Nehari manifold 
$$\cN:=\{u\in E^{\alpha / 2}\setminus\{0\}:\; \J'(u)(u)=0\}.$$
Obviously $\cN$ contains all nontrivial critical points, hence a ground state is the least energy solution.

Note that if $\Gamma\not\equiv 0$, then the nonlinear part of the energy functional 
$$\int_{\R^N} \Big(F(x,u)-\frac1q\Gamma(x)|u|^q\Big)\, dx$$
is sign-changing, moreover $u\mapsto \big(f(x,u)-\Gamma(x)|u|^{q-2}u\big)/|u|^2$ is no longer increasing on $(-\infty,0)$ and $(0,\infty)$.

The classical Schr\"odinger equation (the case $\alpha = 2$) has been studied by many authors; see for instance \cite{AlamaLi, BuffoniJeanStuart, CotiZelati, KryszSzulkin, Rabinowitz:1992, MederskiTMNA2014} and references therein. The fractional case has been also widely investigated in \cite{doO, Binlin, He, Ros, Ambrosio, Chen, Davila, Fall, Davila2, Fall2, Frank2}; see also references therein.

The existence of nontrivial solutions was obtained by S. Secchi in \cite{secchi} for subcritical $f \in \cC^1 (\R^N \times \R)$ satisfying the Ambrosetti-Rabinowitz type condition $0 < \mu F(x,u) < u f(x,u)$ for $\mu > 2$ and coercive potential $V \in \cC^1 (\R^N)$. In \cite{secchi} was also introduced the Nehari manifold method with the classical monotonicity condition: $t \mapsto t^{-1} u f(x, tu)$ is increasing on $(0, \infty)$. M. Cheng proved in \cite{Cheng} that (\ref{eq}) has a nontrivial solution for the subcritical nonlinearity $f(x,u) = |u|^{p-1}u + \omega u$ and coercive potential $V(x) > 1$ for a.e. $x \in \R^N$. He showed also that there is a ground state solution being minimizer on the Nehari manifold for $0 < \omega < \lambda$, where $\lambda = \inf \sigma (\mathcal{A})$ and $\sigma (\mathcal{A})$ is the spectrum of the self-adjoint operator $\mathcal{A} := (-\Delta)^{\alpha / 2} + V(x)$ on $L^2 (\R^N)$. For $f(x,u) = |u|^{p-1}$, where $p$ is a subcritical exponent, there is a positive and spherically symmetric solution (see \cite{Dipierro}). The uniqueness of ground states $Q=Q(|x|)\geq 0$ of an equation $(-\Delta)^{\alpha / 2} Q + Q - Q^{\beta + 1} = 0$ in $\R$ was obtained by R.L. Frank and E. Lenzmann in \cite{Frank}. Recently, S. Secchi proved the existence of radially symmetric solution of $(-\Delta)^{\alpha / 2} u + V(x) u = g(u)$ for $g$ which does not satisfy the Ambrosetti-Rabinowitz condition (\cite{secchiTMNA}). Such a result was known before for $\alpha = 2$ and constant potentials $V$ (\cite{BerestyckiLions}). In our case the nonlinear term depends on $x$, does not satisfy the Ambrosetti-Rabinowitz condition and the classical monotonicity condition is violated.

When $\alpha = 2$, the existence of ground states was obtained in \cite{BieganowskiMederski} with the assumption ($V_\alpha$1). Using the abstract setting provided in \cite{BieganowskiMederski} we are able to extend this result for $0<\alpha<2$ and $q > 2$. In fact, we state the following result.

\begin{Th}\label{ThMain1}
Let $\alpha \in (0,2]$. Suppose that ($V_\alpha$1), ($\Gamma$) (F1)--(F4) are satisfied and $V_{loc} \equiv 0$ or $V_{loc} (x) < 0$ for a.e. $x\in\R^N$. Then \eqref{eq} has a ground state, i.e. there is a nontrivial critical point $u$ of $\J$ such that $\J(u)=\inf_{\cN}\J$. 
\end{Th}

For $V_{loc} \equiv 0$ we provide a multiplicity result, which is new also in the case $\alpha = 2$.

Suppose that $u$ is a solution of \eqref{eq} and $k \in \mathbb{Z}^N$, observe that $u(\cdot - k)$ is also the solution, provided by $V_{loc} \equiv 0$. Therefore all elements of the orbit 
$$
\mathcal{O}(u) := \left\{ u( \cdot - k) \ : \ k \in \mathbb{Z}^N \right\}
$$
of $u$ under the $\mathbb{Z}^N$-action are solutions. Thus, we define that $u_1$ and $u_2$ are geometrically distinct if their orbits satisfy $\mathcal{O}(u_1) \cap \mathcal{O}(u_2) = \emptyset$.

\begin{Th}\label{ThMultiplicity}
Let $\alpha \in (0,2]$. Suppose that ($V_\alpha1$), ($\Gamma$) (F1)--(F4) are satisfied, $V_{loc} \equiv 0$ and suppose that $f$ is odd in $u$. Then \eqref{eq} admits infinitely many pairs $\pm u$ of geometrically distinct solutions.
\end{Th}

Moreover, we investigate the existence of solutions when the potential is coercive.

\begin{Th}\label{ThMain2}
Let $\alpha \in (0,2]$. Suppose that ($V_\alpha$2), ($\Gamma$) (F1)--(F4) are satisfied. Then \eqref{eq} has a ground state, i.e. there is a nontrivial critical point $u$ of $\J$ such that $\J(u)=\inf_{\cN}\J$. 
\end{Th}

Observe that this result does not require such a regularity of the potential and the nonlinear term as \cite{secchi}[Theorem 3.1] -- we do not need the differentiability of the right side of (\ref{eq}) and the potential. Note that the coercive potential has been studied in \cite{Torres}, however our nonlinearity is not of the form $f(u)+h(x)$, where $f$ satisfies the Ambrosetti-Rabinowitz condition.

\section{Preliminary facts}
\label{sect:Preliminary}

Suppose that $E$ is an Hilbert space with respect to the norm $\| \cdot \|$. Let us consider a functional $\cJ : E \rightarrow \R$ of the general form
$$
\cJ(u) = \frac{1}{2} \|u\|^2 - \cI (u),
$$
where $\cI : E \rightarrow \R$ is of $\cC^1$-class. Let us recall a critical point theorem from \cite{BieganowskiMederski}, which is based on the approach of \cite{SzulkinWeth} and \cite{BartschMederski1}.

\begin{Th}[\cite{BieganowskiMederski}]\label{ThSetting}
Suppose that the following conditions hold:
\begin{enumerate}
\item[(J1)] there is $r>0$ such that $a:= \inf_{\|u\|=r} \cJ (u) > \cJ(0) = 0$;
\item[(J2)] there is $q \geq 2$ such that $\cI (t_n u_n) / t_n^q \to \infty$ for any $t_n \to\infty$ and $u_n\to u \neq 0$ as $n\to\infty$;
\item[(J3)] for $t \in (0,\infty) \setminus \{1\}$ and $u \in \cN$
$$
\frac{t^2-1}{2} \cI'(u)(u) - \cI(tu)+\cI(u) < 0;
$$
\item[(J4)] $\cJ$ is coercive on $\cN$.
\end{enumerate}
Then $\inf_\cN \cJ > 0$ and there exists a bounded minimizing sequence for $\cJ$ on $\cN$, i.e. there is a sequence $(u_n) \subset \cN$ such that $\cJ(u_n) \to \inf_\cN \cJ$ and $\cJ'(u_n) \to 0$.
\end{Th} 

This setting allows us to find a bounded minimizing sequence for $\cJ$ on $\cN$. We only need to check whether (J1)--(J4) are satisfied -- we follow arguments from \cite{BieganowskiMederski}, we provide details below.

\begin{Lem}
Let $E = E^{\alpha / 2}$ and take
$$
\cI(u) = \int_{\R^N} \Big(F(x,u(x))-\frac1q\Gamma(x)|u(x)|^q\Big)\, dx.
$$
Suppose that ($\Ga$), (F1)--(F4) are satisfied. Suppose also that ($V_\alpha$1) or ($V_\alpha$2) hold. Then (J1)--(J4) hold.
\end{Lem}

\begin{proof}
\begin{itemize}
	\item[(J1)] Fix $\varepsilon > 0$. Observe that (F1) and (F2) implies that $F(x,u) \leq \varepsilon |u|^2 + C_\varepsilon |u|^p$ for some $C_\varepsilon > 0$. Therefore 
	$$
	\int_{\R^N} F(x,u) \, dx - \int_{\R^N} \frac{1}{q} \Ga (x) |u|^q \, dx \leq \int_{\R^N} F(x,u) \, dx \leq C(\varepsilon \|u\|^2 + C_\varepsilon \|u\|^p),
	$$
for some constant $C > 0$ provided by the Sobolev embedding theorem. Thus there is $r > 0$ such that 
$$
\int_{\R^N} F(x,u) \, dx - \int_{\R^N} \frac{1}{q} \Ga (x) |u|^q \, dx \leq \frac{1}{4} \|u\|^2
$$
for $\|u\| \leq r$.
Therefore
$$
\cJ (u) \geq \frac{1}{4} \|u\|^2 = \frac{1}{4} r^2 > 0
$$
for $\|u\| = r$.
	\item[(J2)] By (F3) and Fatou's lemma we get
	$$
	\cI (t_n u_n) / t_n^q = \int_{\R^N} \frac{F(x, t_n u_n)}{t_n^q} \, dx - \frac{1}{q} \int_{\R^N} \Ga (x) |u_n|^q \, dx \to \infty.
	$$
	\item[(J3)] Fix $u \in \cN$ and consider
$$
\psi (t) = \frac{t^2-1}{2} \cI'(u)(u) - \cI(tu) + \cI(u)
$$
for $t \geq 0$. Then $\psi(1) = 0$ and
$$
\frac{d\psi(t)}{dt} = \int_{\R^N} tf(x,u)u - f(x,tu)u \, dx + (t^{q-1} - t) \int_{\R^N} \Ga(x) |u|^q \, dx.
$$
Since $u \in \cN$
$$
\int_{\R^N} f(x,u)u \, dx - \int_{\R^N} \Ga (x) |u|^q \, dx = \|u\|^2 > 0.
$$
Therefore, for $t>1$ we have
$$
\frac{d\psi(t)}{dt} < \int_{\R^N} t^{q-1} f(x,u)u - f(x,tu)u \, dx = t^{q-1} \int_{\R^N} f(x,u)u - \frac{f(x,tu)u}{t^{q-1}} \, dx < 0,
$$
by (F4). Similarly $\frac{d\psi(t)}{dt} > 0$ for $t < 1$. Therefore $\psi(t) < \psi(1) = 0$ for $t \neq 1$, i.e.
$$
\frac{t^2-1}{2} \cI'(u)(u) - \cI(tu) + \cI(u) < 0.
$$

\item[(J4)] Let $(u_n) \subset \cN$ be a sequence such that $\|u_n\| \to \infty$ as $n\to\infty$. (F3) implies that
$$
f(x,u)u = q\int_0^u \frac{f(x,u)}{u^{q-1}}s^{q-1}\,ds \geq q\int_0^u \frac{f(x,s)}{s^{q-1}}s^{q-1}\,ds=q F(x,u)
$$
for $u\geq 0$ and similarly $f(x,u)u\geq q F(x,u)$ for $u<0$.
Therefore
\begin{eqnarray*}
\cJ(u_n) &=& \frac{1}{2} \|u_n\|^2 - \int_{\R^N} F(x, u_n) \, dx + \frac{1}{q} \int_{\R^N} \Ga (x) |u_n|^q \, dx = \\
&=& \left( \frac{1}{2} - \frac{1}{q} \right) \|u_n\|^2 + \int_{\R^N} \frac{1}{q} f(x, u_n) u_n - F(x, u_n) \, dx \geq \\
&\geq& \left( \frac{1}{2} - \frac{1}{q} \right) \|u_n\|^2 \to \infty
\end{eqnarray*}
\end{itemize}
as $n\to\infty$, since $q > 2$.
\end{proof}

\section{Palais-Smale sequences decomposition}
\label{sect:Splitting}

The main theorem in this section generalizes the decomposition result form \cite{BieganowskiMederski}[Theorem 4.1]. We consider the functional $\cJ : H^{\alpha / 2} (\R^N) \rightarrow \R$ of the form
$$
\cJ (u) = \frac{1}{2} \|u\|^2 - \int_{\R^N} G(x,u) \, dx.
$$
We assume that $(V_\alpha 1)$ holds and therefore we may consider the following norm
$$
\|u\|^2 = \int_{\R^N} |\xi|^\alpha |\hat{u}(\xi)|^2 \, d\xi + \int_{\R^N} V(x) |u|^2 \, dx.
$$
We suppose that $G(x,u) = \int_0^u g(x,s) \, ds$, where $g : \R^N \times \R \rightarrow \R$ satisfies:
\begin{itemize}
\item[($G1$)] $g(\cdot, u)$ is measurable and $\mathbb{Z}^N$-periodic in $x \in \R^N$, $g(x,\cdot)$ is continuous in $u \in \R$ for a.e. $x \in \R^N$;
\item[($G2$)] $g(x,u) = o(|u|)$ as $|u| \to 0^+$ uniformly in $x \in \R^N$;
\item[($G3$)] there exists $2 < r < 2_\alpha^*$ such that $\lim_{|u| \to \infty} g(x,u)/|u|^{r-1} = 0$ uniformly in $x \in \R^N$;
\item[($G4$)] for each $a<b$ there is a constant $c>0$ such that  $|g(x,u)|\leq c$ for a.e. $x\in\R^N$ and $a\leq u\leq b$.
\end{itemize} 
We will also denote
$$
\cJ_{per} (u) = \cJ (u) - \int_{\R^N} V_{loc} (x) |u|^2 \, dx.
$$

\begin{Th}\label{ThDecomposition}
Suppose that ($G1$)--($G4$) and ($V_\alpha 1$) hold. Let $(u_n)$ be a bounded Palais-Smale sequence for $\cJ$. Then passing to a subsequence of $(u_n)$, there is an integer $\ell > 0$ and sequences $(y_n^k) \subset \mathbb{Z}^N$, $w^k \in H^{\alpha / 2} (\R^N)$, $k = 1, \ldots, \ell$ such that:
\begin{itemize}
\item[$(a)$] $u_n \rightharpoonup u_0$ and $\cJ' (u_0) = 0$;
\item[$(b)$] $|y_n^k| \to \infty$ and $|y_n^k - y_n^{k'}| \to \infty$ for $k \neq k'$;
\item[$(c)$] $w^k \neq 0$ and $\cJ_{per}'(w^k) = 0$ for each $1 \leq k \leq \ell$;
\item[$(d)$] $u_n - u_0 - \sum_{k=1}^\ell w^k (\cdot - y_n^k) \to 0$ in $H^{\alpha / 2} (\R^N)$ as $n \to \infty$;
\item[$(e)$] $\cJ (u_n) \to \cJ(u_0) + \sum_{k=1}^\ell \cJ_{per} (w^k)$.
\end{itemize}
\end{Th}

\begin{Rem} \label{Rem:Epsilon}
Note that (G2)--(G4) imply that for every $\varepsilon > 0$ there is $C_\varepsilon > 0$ such that
$$
|g(x,u)| \leq \varepsilon |u| + C_\varepsilon |u|^{r-1}
$$
for any $u \in \R$ and a.e. $x \in \R^N$.
\end{Rem}


The proof of Theorem \ref{ThDecomposition} is, in fact, the reformulation of the proof of \cite{BieganowskiMederski}[Theorem 4.1]. We use \cite{secchi}[Lemma 2.4] instead of the classical Lions' lemma and we observe that the classical norm on $H^{\alpha / 2} (\R^N)$
$$
\| u \|_{H^{\alpha / 2}}^2 = \int_{\R^N} |\xi|^\alpha |\hat{u}(\xi)|^2 \,d\xi + \int_{\R^N} |u(x)|^2 \,dx
$$
is $\mathbb{Z}^N$-invariant. Then the remaining arguments in proof of \cite{BieganowskiMederski}[Theorem 4.1] are the same.

\section{Existence of solutions for the sum of periodic and localized potential}
\label{sect:Existence}

\begin{Lem}
Let $\alpha \in (0,2)$ and denote by 
$$
\| u \|_{H^{\alpha / 2}}^2 = \int_{\R^N} |\xi|^\alpha |\hat{u}(\xi)|^2 \,d\xi + \int_{\R^N} |u(x)|^2 \,dx
$$
the classical norm on the fractional Sobolev space $H^{\alpha / 2} (\R^N)$. Suppose that $(V_\alpha1)$ hold. Then the following norm
$$
\| u\|^2 = \int_{\R^N} |\xi|^\alpha |\hat{u}(\xi)|^2 \,d\xi + \int_{\R^N} V(x) |u(x)|^2 \,dx
$$
is equivalent to $\| \cdot \|_{H^{\alpha / 2}}$.
\end{Lem}

\begin{proof}
Indeed - observe that
$$
\|u\|^2 \leq \int_{\R^N} |\xi|^\alpha |\hat{u}(\xi)|^2 \,d\xi + |V|_\infty \int_{\R^N} |u(x)|^2 \,dx \leq \max\{ 1, |V|_\infty \} \| u \|_{H^{\alpha / 2}}^2 
$$
and
$$
\|u\|^2 \geq \int_{\R^N} |\xi|^\alpha |\hat{u}(\xi)|^2 \,d\xi + V_0 \int_{\R^N} |u(x)|^2 \,dx \geq \min \{1, V_0 \} \| u \|_{H^{\alpha / 2}}^2.
$$
\end{proof}

\begin{Rem}
The norm equivalence is also true for $\alpha = 2$.
\end{Rem}

The above lemma implies that $E^{\alpha / 2}$ coincides with $H^{\alpha / 2}(\R^N)$. Thus our functional $\cJ : H^{\alpha / 2}(\R^N) \rightarrow \R$ has the form
$$
\cJ(u) = \frac12 \|u\|^2 - \cI(u),
$$
where
$$
\cI(u) = \int_{\R^N} \Big(F(x,u)-\frac1q\Gamma(x)|u|^q\Big)\, dx.
$$
The Nehari manifold is given by
$$
\cN = \{ u \in H^{\alpha / 2} (\R^N) \setminus \{0\} \ : \ \|u\|^2 = \cI'(u)(u) \}.
$$

Now we are ready to prove our first result.

\begin{altproof}{Theorem \ref{ThMain1}}%

By Theorem \ref{ThSetting} we find a bounded minimizing sequence on $\cN$, i.e. sequence $(u_n) \subset \cN$ such that $\cJ(u_n) \to c := \inf_\cN \cJ > 0$ and $\cJ'(u_n) \to 0$. By Theorem \ref{ThDecomposition} we have that
$$
\cJ (u_n) \to \cJ(u_0) + \sum_{k=1}^\ell \cJ_{per} (w^k),
$$
where $w^k$ are critical points of the periodic part of the functional $\cJ$, i.e. critical points of
$$
\cJ_{per} (u) = \cJ (u) - \int_{\R^N} V_{loc}(x) u^2 \, dx.
$$
Suppose that $V_{loc} \equiv 0$, i.e. $\cJ = \cJ_{per}$. If $u_0 = 0$, we have
$$
c \leftarrow \cJ (u_n) \to \sum_{k=1}^\ell \cJ_{per} (w^k) \geq \ell c
$$
and therefore $\ell = 1$ and $w^1 \neq 0$ is a ground state. If $u_0 \neq 0$ we have
$$
c \leftarrow \cJ (u_n) \to \cJ (u_0) + \sum_{k=1}^\ell \cJ_{per} (w^k) \geq (\ell+1) c
$$
and therefore $\ell = 0$ and $\cJ (u_n) \to \cJ (u_0) = c$, so $u_0$ is a ground state.

Suppose that $V_{loc} (x) < 0$ for a.e. $x \in \R^N$. Then $\inf_{\cN_{per}} \cJ_{per} = c_{per} > c$, where
$$
\cN_{per} = \left\{ u \neq 0 \ : \ \cJ_{per}' (u)(u) = 0 \right\}.
$$ 
Suppose that $u_0 = 0$. Therefore
$$
c \leftarrow \cJ (u_n) \to \sum_{k=1}^\ell \cJ_{per} (w^k) \geq \ell c_{per} > \ell c
$$
and $\ell = 0$. Thus $\cJ (u_n) \to 0 = c$ -- a contradiction. Therefore $u_0 \neq 0$ and observe that
$$
c \leftarrow \cJ (u_n) \to \cJ(u_0) + \sum_{k=1}^\ell \cJ_{per} (w^k) \geq c + \ell c_{per},
$$
and $\ell = 0$. It means that $\cJ (u_n) \to \cJ (u_0) = c$ and $u_0$ is a ground state.

\end{altproof}

\section{Multiplicity result for the periodic potential}
\label{sect:Multiplicity}

Put $c = \inf_{\cN} \cJ > 0$ and $\beta = \inf_{\cN} \|u\| > 0$. Theorem \ref{ThMain1} provides that $c$ is attained at some function in $\cN$. By $\tau_k$ we denote the $\mathbb{Z}^N$-action on $H^{\alpha / 2} (\R^N)$, i.e.
$$
\tau_k u = u(\cdot - k).
$$
Obviously, $\tau_k \tau_{-k} u = u$.

\begin{Lem}\label{Lem:ScalarProductInvariance}
There holds
$$
\langle \tau_k u, v \rangle = \langle u, \tau_{-k} v \rangle
$$
for every $u,v \in H^{\alpha / 2} (\R^N)$ and $k \in \mathbb{Z}^N$.
\end{Lem}

\begin{proof}
For $\alpha = 2$ the observation is trivial. Let $\alpha < 2$. Note that
$$
\mathcal{F} (\tau_k u) (\xi) = e^{-2\pi i \xi \cdot k} \mathcal{F} ( u) (\xi)
$$
for every $k \in \Z^N$ and $\xi \in \R^N$.
Thus
$$
\mathcal{F}(\tau_k u) \cdot \overline{\mathcal{F}(v)} = e^{-2\pi i \xi \cdot k} \mathcal{F} ( u) \overline{\mathcal{F}(v)} = \mathcal{F}(u) \cdot \overline{ e^{-2\pi i \xi \cdot (-k)} \mathcal{F}(v)} = \mathcal{F}(u) \cdot \overline{\mathcal{F}(\tau_{-k} v)}.
$$
Therefore
$$
\int_{\R^N} |\xi|^\alpha \widehat{\tau_k u} (\xi) \overline{\hat{v} (\xi)} \, d\xi = \int_{\R^N} |\xi|^\alpha \hat{u} (\xi) \overline{\widehat{\tau_{-k} v} (\xi)} \, d\xi.
$$
Obviously, using a change of variables $x \mapsto x + k$
$$
\int_{\R^N} V(x) (\tau_k u) v \, dx = \int_{\R^N} V(x) u (\tau_k v) \, dx
$$ 
and we conclude.
\end{proof}

\begin{Rem}
For given $k \in \mathbb{Z}^N$, let us consider $\tau_k$ as an operator $\tau_k : H^{\alpha / 2} (\R^N) \rightarrow H^{\alpha / 2} (\R^N)$. Then obviously $\tau_k$ is linear. Moreover
$$
\|\tau_k u\| = \|u\|,
$$
thus $\tau_k$ is a bounded operator and $\|\tau_k\| = 1$. Thus we may consider an adjoint operator $\tau_k^* : H^{\alpha / 2} (\R^N) \rightarrow H^{\alpha / 2} (\R^N)$. The above lemma says that
$$
\tau_k^* = \tau_{-k}.
$$
Moreover $\tau_k$ is an isomorphism and $\tau_k^{-1} = \tau_{-k} = \tau_k^*$. Thus $\tau_k$ is an orthogonal operator.
\end{Rem}

\begin{Lem}
Let $\alpha \in (0,2]$. The functional $\cJ$ is $\mathbb{Z}^N$-invariant.
\end{Lem}

\begin{proof}
Let us start with the trivial observation that if $u \in H^{\alpha / 2} (\R^N)$, and $w \in \mathcal{O}(u)$, then by Lemma \ref{Lem:ScalarProductInvariance}
$$
\| u \| = \| w\|.
$$
Indeed -- $w = \tau_k u$ for some $k \in \mathbb{Z}^N$. Then
$$
\|w\|^2 = \langle w, w \rangle = \langle \tau_k u, \tau_k u \rangle = \langle u, \tau_{-k} \tau_k u \rangle = \langle u, u \rangle = \|u\|^2.
$$
Then
$$
\cJ (w) = \frac{1}{2} \|w\|^2 - \int_{\R^N} F(x,w) \, dx + \frac{1}{q} \int_{\R^N} \Gamma(x) |w|^{q} \, dx.
$$
Changing variables in the integrals $x \mapsto x+k$ and the $\mathbb{Z}^N$-periodicity of $F$ and $\Gamma$ in $x$ gives
$$
\int_{\R^N} F(x,w) \, dx = \int_{\R^N} F(x,u) \, dx, \quad \frac{1}{q} \int_{\R^N} \Gamma(x) |w|^{q} \, dx = \frac{1}{q} \int_{\R^N} \Gamma(x) |u|^{q} \, dx.
$$
Therefore $\cJ (w) = \cJ (u)$. 

\end{proof}

\begin{Lem}\label{Lem:NehariInvariance}
$\cN$ is $\mathbb{Z}^N$-invariant.
\end{Lem}

\begin{proof}
Suppose that $u \in \cN$. Then
$$
\cJ'(\tau_k u) (\tau_k u) = \| \tau_k u \|^2 - \int_{\R^N} f(x,\tau_k u)\tau_k u \, dx + \int_{\R^N} \Ga(x) |\tau_k u|^q \, dx.
$$
By the $\mathbb{Z}^N$-periodicity of $f$ and $\Gamma$ we have
$$
\int_{\R^N} f(x,\tau_k u)\tau_k u \, dx = \int_{\R^N} f(x, u) u \, dx, \quad \int_{\R^N} \Ga(x) |\tau_k u|^q \, dx=\int_{\R^N} \Ga(x) |u|^q \, dx.
$$
Moreover, by Lemma \ref{Lem:ScalarProductInvariance} $\| \tau_k u \| = \|u\|$ and finally $\cJ'(\tau_k u) (\tau_k u) = 0$, which shows that $\mathcal{O}(u) \subset \cN$.
\end{proof}

\begin{Rem}
Lemma \ref{Lem:ScalarProductInvariance} implies that the unit sphere $S^1$ is $\mathbb{Z}^N$-invariant.
\end{Rem}

Recall that for each $u \in H^{\alpha / 2} (\R^N)$ there is a unique number $t(u) > 0$ such that $t(u)u \in \cN$ and moreover the function $m : S^1 \rightarrow \cN$ given by $m(u)=t(u)u$ is a homeomorphism (see \cite{BieganowskiMederski}). The inverse $m^{-1} : \cN \rightarrow S^1$ is given by $m^{-1} (u) = u/\|u\|$.

\begin{Lem}\label{LemEquivariance}
Functions:
\begin{itemize}
\item $m : S^1 \rightarrow \cN$,
\item $m^{-1} : \cN \rightarrow S^1$,
\item $\nabla \cJ : H^{\alpha / 2} (\R^N) \rightarrow H^{\alpha / 2} (\R^N)$,
\item $ \nabla(\cJ \circ m) : S^1 \rightarrow H^{\alpha / 2} (\R^N)$
\end{itemize}
are $\mathbb{Z}^N$-equivariant.
\end{Lem}

\begin{proof} $ $ 
\begin{itemize}
\item \textbf{Equivariance of} $m$. \\
Take $u \in S^1$, since $\|u\| = \|\tau_k u\|$, we have $\tau_k u \in S^1$. There is unique number $t=t(u)>0$ such that $m(u) = t(u)u \in \cN$. We claim that $t(u) \tau_k u \in \cN$. Indeed
$$
t(u) \tau_k u = \tau_k \left( t(u)u \right) = \tau_k m(u) \in \cN,
$$
by Lemma \ref{Lem:NehariInvariance}. Thus 
$$
m(\tau_k u) = t(u) \tau_k u = \tau_k m(u).
$$
\item \textbf{Equivariance of} $m^{-1}$. \\
Let $u \in \cN$. By Lemma \ref{Lem:NehariInvariance} we have that $\tau_k u \in \cN$. Observe that
$$
m^{-1} (\tau_k u) = \frac{\tau_k u}{\| \tau_k u\|} = \frac{\tau_k u}{\|u\|} = \tau_k \left( \frac{u}{\|u\|} \right) = \tau_k m^{-1} (u).
$$
\item \textbf{Equivariance of} $\nabla \cJ$. \\ Take $u,v \in H^{\alpha / 2} (\R^N)$. Then
\begin{eqnarray*}
& & \langle \nabla \cJ ( \tau_k u), v \rangle = \cJ ' (\tau_k u) (v) \\
&=& \langle \tau_k u, v \rangle - \int_{\R^N} f(x,\tau_k u) v \, dx + \int_{\R^N} \Gamma(x) |\tau_k u|^{q-2} (\tau_k u) v \, dx
  \\
&=& \langle u, \tau_{-k} v \rangle - \int_{\R^N} f(x,u) \tau_{-k} v \, dx + \int_{\R^N} \Gamma(x) |u|^{q-2} u (\tau_{-k}v) \, dx  \\
&=& \langle \nabla \cJ (u), \tau_{-k}v \rangle = \langle \tau_k \nabla \cJ (u), v \rangle
\end{eqnarray*}
Therefore $\nabla \cJ (\tau_k u) = \tau_k \nabla \cJ (u)$ for every $u \in H^{\alpha /2}(\R^N)$.
\item \textbf{Equivariance of} $\nabla(\cJ \circ m)$. \\
For $u \in S^1$ and $z \in T_u S^1$, by the abstract setting in \cite{BieganowskiMederski} we know that
$$
(\cJ \circ m)'(u)(z) = \| m(u) \| \cJ'(m(u))(z).
$$
So take $u \in S^1$ and $z \in T_{\tau_k u} S^1$. Then $\tau_{-k} z \in T_u S^1$. Therefore
\begin{eqnarray*}
& & \langle \nabla (\cJ \circ m) (\tau_k u), z \rangle = (\cJ \circ m)'(\tau_k u)(z) = \| m(\tau_k u) \| \cJ'(m(\tau_k u))(z) \\
&=& \|\tau_k m(u) \| \cJ' (\tau_k (m(u))(z) = \| m(u) \| \langle \nabla \cJ (\tau_k (m(u)), z \rangle \\
&=& \|m(u)\| \langle \tau_k \nabla \cJ (m(u)), z \rangle = \|m(u)\| \langle \nabla \cJ (m(u)), \tau_{-k} z \rangle \\
&=& (\cJ \circ m)'(u)(\tau_{-k} z) = \langle \nabla (\cJ \circ m) (u), \tau_{-k} z \rangle = \langle \tau_{k} \nabla (\cJ \circ m) (u),  z \rangle.
\end{eqnarray*}
Thus $\tau_{k} \nabla (\cJ \circ m) (u) = \nabla (\cJ \circ m) (\tau_k u)$ for every $u \in S^1$.
\end{itemize}

\end{proof}

\begin{Lem}
Function $m^{-1} : \cN \to S^1$ is Lipschitz continuous.
\end{Lem}

\begin{proof}
Let $u,v \in \cN$. Observe that
\begin{eqnarray*}
\| m^{-1} (u) - m^{-1} (v) \| &=& \left\| \frac{u-v}{\|u\|} + \frac{v \|v\|-v\|u\|}{\|u\| \cdot \|v\|} \right\| = \left\| \frac{u-v}{\|u\|} + \frac{v (\|v\|-\|u\|)}{\|u\| \cdot \|v\|} \right\| \leq \\
&\leq& \frac{\|u-v\|}{\|u\|} + \frac{\left| \|v\| - \|u\| \right|}{\|u\|} \leq \frac{2\|u-v\|}{\|u\|} \leq \frac{2}{\beta} \|u-v\|.
\end{eqnarray*}
\end{proof}

To show Theorem \ref{ThMultiplicity} we employ the method introduced by A. Szulkin and T. Weth in \cite{SzulkinWeth}. Put $\mathscr{C} = \{ u \in S^1 \ : \ (\mathcal{J} \circ m)'(u) = 0 \}$. Choose a set $\mathcal{F} \subset \mathscr{C}$ such that $\mathcal{F} = - \mathcal{F}$ and for each orbit $\mathcal{O}(w)$ there is unique representative $v \in \mathcal{F}$. To show Theorem \ref{ThMain2} we need to show that $\mathcal{F}$ is infinite. Suppose by a contradiction that
$$
\mathcal{F} \ \mbox{is finite.}
$$
Put
$$
\kappa = \inf \{ \|v-w\| \ : \ v,w \in \mathscr{C}, v\neq w \}
$$
and note that $\kappa > 0$ -- see \cite{SzulkinWeth}[Lemma 2.13]. 



The following lemma has the crucial role in the proof of the multiplicity result and originally has been proven in \cite{SzulkinWeth}[Lemma 2.14]. Since the nonlinear term is sign-changing we need a slight modification of the proof.

\begin{Lem}\label{LemDiscreteness}
Let $d \geq c$. If $(v_n^1), (v_n^2) \subset S^1$ are two Palais-Smale sequences for $\mathcal{J} \circ m$ such that $\mathcal{J}(m(v_n^i)) \leq d$, $i=1,2$. then 
$$
\|v_n^1 - v_n^2 \| \to 0
$$
or
$$
\limsup_{n\to\infty} \|v_n^1 - v_n^2\| \geq \rho (d) > 0
$$
where $\rho(d)$ depends only on $d$, but not on the particular choise of sequences.
\end{Lem}

\begin{proof}
We have that $u_n^i = m(v_n^i)$, $i=1,2$, are Palais-Smale sequences for $\cJ$. Moreover they are bounded in $H^{\alpha / 2}(\R^N)$, since $\cJ$ is coercive on $\cN$. Therefore $(u_n^1)$ and $(u_n^2)$ are bounded in $L^2 (\R^N)$, say $|u_n^1|_2 + |u_n^2|_2 \leq M$ for some $M>0$.

\begin{itemize}
\item \textbf{Case 1:} Assume that $|u_n^1 - u_n^2 |_p \to 0$. Fix $\varepsilon > 0$. Then, by (F1), (F2) we have
\begin{eqnarray*}
\| u_n^1 - u_n^2\|^2 &=& \mathcal{J}'(u_n^1)(u_n^1-u_n^2) - \mathcal{J}'(u_n^2)(u_n^1-u_n^2)  \\ 
&+& \int_{\R^N} \left[ f(x,u_n^1) - f(x,u_n^2) \right] (u_n^1 - u_n^2) \, dx \\
&-& \int_{\R^N} \Gamma(x) \left[ |u_n^1|^{q-2} u_n^1 - |u_n^2|^{q-2} u_n^2 \right] (u_n^1-u_n^2) \, dx \\
&\leq& \varepsilon \| u_n^1 - u_n^2\| \\
&+& \int_{\R^N} \left[ \varepsilon (|u_n^1| + |u_n^2|) + C_\varepsilon (|u_n^1|^{p-1} + |u_n^2|^{p-1}) \right] |u_n^1 - u_n^2| \, dx \\
&-& \int_{\R^N} \Gamma(x) \left[ |u_n^1|^{q-2} u_n^1 - |u_n^2|^{q-2} u_n^2 \right] (u_n^1-u_n^2) \, dx \\
&\leq& (1+C_0) \varepsilon \|u_n^1 - u_n^2\| + D_\varepsilon |u_n^1 - u_n^2|_p + C_1 |\Ga|_\infty |u_n^1-u_n^2|_q^q
\end{eqnarray*}
for each $n \geq n_\varepsilon$ and some constants $C_0, C_1, D_\varepsilon > 0$. By our assumption we have that 
$$
D_\varepsilon |u_n^1 - u_n^2|_p \to 0.
$$
Observe that
$$
C_1 |\Ga|_\infty |u_n^1 - u_n^2|_q^q \leq C_1 |\Ga|_\infty |u_n^1 - u_n^2|_2^{\theta q} |u_n^1-u_n^2|_p^{(1-\theta)q},
$$
where $\theta \in (0,1)$ is such that
$$
\frac{1}{q} = \frac{\theta }{2} + \frac{1-\theta}{p}.
$$
Thus 
$$
C_1 |\Ga|_\infty |u_n^1 - u_n^2|_q^q \leq C_1 |\Ga|_\infty M^{\theta q} |u_n^1-u_n^2|_p^{(1-\theta)q} \to 0.
$$
Finally
\begin{eqnarray*}
\limsup_{n\to\infty} \| u_n^1 - u_n^2\|^2 &\leq& \limsup_{n\to\infty} (1+C_0) \varepsilon \|u_n^1 - u_n^2\| \\
&+& \limsup_{n\to\infty} D_\varepsilon |u_n^1 - u_n^2|_p \\
 &+& \limsup_{n\to\infty} C_1 |\Ga|_\infty |u_n^1-u_n^2|_q^q \\
 &=& (1+C_0) \varepsilon \limsup_{n\to\infty}  \|u_n^1 - u_n^2\|
\end{eqnarray*}
for every $\varepsilon > 0$. Therefore $\lim_{n\to\infty} \|u_n^1 - u_n^2\| = 0$. Finally
$$
\| v_n^1 - v_n^2 \| = \| m^{-1} (u_n^1) - m^{-1} (u_n^2) \| \leq L \| u_n^1 - u_n^2 \| \to 0,
$$
where $L > 0$ is a Lipschitz constant for $m^{-1}$.

\item \textbf{Case 2:} Assume that $|u_n^1 - u_n^2|_p \not\to 0$. By the Lions lemma there are $y_n \in \R^N$ such that
$$
\int_{B(y_n, 1)} |u_n^1 - u_n^2|^2 \, dx = \max_{y \in \R^N} \int_{B(y, 1)} |u_n^1 - u_n^2|^2 \, dx \geq \varepsilon
$$
for some $\varepsilon > 0$. In view of Lemma \ref{LemEquivariance} we can assume that $(y_n)$ is bounded. Therefore, up to a subsequence we have
$$
u_n^1 \rightharpoonup u^1, \quad u_n^2 \rightharpoonup u^2
$$
where $u^1 \neq u^2$ and $\mathcal{J}'(u^1) = \mathcal{J}'(u^2) = 0$, and
$$
\| u_n^1 \| \to \alpha^1, \quad \| u_n^2 \| \to \alpha^2,
$$
where $\beta \leq \alpha^i \leq \nu (d) = \sup \{ \|u\| \ : \ u \in \cN, \ \mathcal{J}(u) \leq d \}$, $i=1,2$. Suppose that $u^1 \neq 0$ and $u^2 \neq 0$. Therefore $u^i \in \cN$ for $i=1,2$. Moreover
$$
v^i = m^{-1} (u^i) \in S^1, \quad i=1,2
$$
and $v^1 \neq v^2$. Then
$$
\liminf_{n\to\infty} \|v_n^1 - v_n^2\| = \liminf_{n\to\infty} \left\| \frac{u_n^1}{\| u_n^1\|} - \frac{u_n^2}{\| u_n^2\|} \right\| \geq \left\| \frac{u^1}{\alpha^1} - \frac{u^2}{\alpha^2} \right\| = \| \beta_1 v_1 - \beta_2 v_2 \|,
$$
where $\beta_i = \frac{\|u^i\|}{\alpha^i} \geq \frac{\beta}{\nu(d)}$, $i=1,2$. Of course $\|v^1\|=\|v^2\|=1$. Therefore
$$
\liminf_{n\to\infty} \|v_n^1 - v_n^2\| \geq \| \beta_1 v^1 - \beta_2 v^2\| \geq \min_{i=1,2} \{ \beta_i \} \|v^1 - v^2\| \geq \frac{\beta \kappa}{\nu(d)}.
$$
If $u^2 = 0$, then $u^1 \neq u^2 = 0$. Therefore
$$
\liminf_{n\to\infty} \|v_n^1 - v_n^2\| = \liminf_{n\to\infty} \left\| \frac{u_n^1}{\| u_n^1\|} - \frac{u_n^2}{\| u_n^2\|} \right\| \geq \left\| \frac{u^1}{ \alpha^1} - \frac{u^2}{\alpha^2} \right\| =  \left\| \frac{u^1}{ \alpha^1} \right\| \geq \frac{\beta}{\nu(d)}.
$$
The case $u^1=0$ is similar, the proof is completed.
\end{itemize}

\end{proof}

\begin{altproof}{Theorem \ref{ThMultiplicity}}
The unit sphere $S^1 \subset H^{\alpha / 2} (\R^N)$ is a Finsler $C^{1,1}$-manifold and by \cite{Struwe}[Lemma II.3.9], $\cJ \circ m : S^1 \rightarrow \R$ admits a pseudo-gradient vector field. The obtained discreteness of Palais-Smale sequences (Lemma \ref{LemDiscreteness}) allows us to repeat the proof of Lemma 2.15, Lemma 2.16 and Theorem 1.2 from \cite{SzulkinWeth} in our case. We show that for every $k \in \mathbb{N}$ there is $u \in S^1$ such that $(\cJ \circ m)'(u) = 0$ and $\mathcal{J} (m(u)) = c_k$, where 
$$
c_k = \inf \{ d \in \mathbb{R} \ : \ \gamma \left( \{ v \in S^1 \ : \ \mathcal{J}(m(v)) \leq d \} \right) \geq k \}
$$
is the Lusternik-Schnirelmann value and $\gamma$ denotes the Krasnoselskii genus (see \cite{Struwe}). Moreover $c_k < c_{k+1}$, thus we get a contradiction.
\end{altproof}

\section{Existence of solutions for the coercive potential}
\label{sect:ExistenceCoercive}



We will need the following form of the Sobolev-Gagliardo-Nirenberg inequality.

\begin{Lem}[Proposition II.3 in \cite{secchi}]\label{LemSGN}
Let $r > 1$. Then there is a positive constant $C > 0$, such that for every function $u \in H^{\alpha / 2} (\R^N)$ there holds
$$
|u|_{r+1}^{r+1} \leq C \| u\|_{H^{\alpha / 2}}^{\frac{(r-1)N}{\alpha}} |u|_2^{r+1-\frac{(r-1)N}{\alpha}}.
$$
\end{Lem}

In \cite{Cheng} it was proven that $E^{\alpha / 2}$ is compactly embedded into $L^{r} (\R^N)$ for $N > \alpha$ and $2 \leq r < 2^*_\alpha$. However we will show that the method introduced by S. Secchi in \cite{secchi} may be applied in this case.

\begin{altproof}{Theorem \ref{ThMain2}}

In view of Theorem \ref{ThSetting}, we obtain a bounded minimizing sequence $(u_n) \subset \cN$, i.e.
$$
\cJ(u_n) \to \inf_\cN \cJ =:c, \quad \cJ ' (u_n) \to 0.
$$
Therefore we may suppose that $u_n \rightharpoonup u_0$ in $E^{\alpha / 2}$ and $u_n \to u_0$ in $L^r_{loc} (\R^N)$ for $1 \leq r < 2^*_\alpha$. It is a classical argument that $u_0$ is a weak solution to our problem. We only need to check whether $u_0 \neq 0$. To show this, observe that for $n \geq n_0$
\begin{eqnarray*}
\frac{c}{2} &\leq& \cJ (u_n) = \cJ(u_n) - \frac{1}{2} \cJ'(u_n)(u_n) \\
&=&  \frac{1}{2} \int_{\R^N} \left( f(x, u_n) u_n - 2 F(x,u_n) \right) \, dx - \left( \frac{1}{2} - \frac{1}{q} \right) \int_{\R^N} \Ga (x) |u|^{q} \, dx \\
&\leq& \frac{1}{2} \int_{\R^N} \left( f(x, u_n) u_n - 2 F(x,u_n) \right) \, dx \leq \frac{1}{2} \int_{\R^N}  f(x, u_n) u_n \, dx  \\
&\leq& \frac{1}{2} \int_{\R^N} \left( \varepsilon |u_n|^2 + C_\varepsilon |u_n|^{p} \right) \, dx,
\end{eqnarray*}
where $n_0 \geq 1$ is large enough. Therefore
$$
\frac{c}{2} \leq \frac{\varepsilon}{2} |u_n|_2^2 + \frac{C_\varepsilon}{2} |u_n|_p^p.
$$
By the Sobolev-Gagliardo-Niremberg inequality (Lemma \ref{LemSGN}) we obtain
$$
\frac{c}{2} \leq \frac{\varepsilon}{2} |u_n|_2^2 + \frac{C \cdot C_\varepsilon}{2} \| u_n\|_{H^{\alpha / 2} (\R^N)}^{\frac{(p-2)N}{\alpha}} \cdot |u_n|_2^{p - \frac{(p-2)N}{\alpha}}.
$$
Observe that the boudedness of $(u_n)$ in $E^{\alpha / 2}$ with respect to the norm $\| \cdot \|$ implies the boundedness of $(u_n)$ in $H^{\alpha / 2} (\mathbb{R}^N)$ with respect to the classical norm $\| \cdot \|_{H^{\alpha / 2} (\R^N)}$. Therefore $\| u_n \|_{H^{\alpha / 2} (\R^N)} \leq D$ for some $D > 0$. Thus
$$
\frac{c}{2} \leq \frac{\varepsilon}{2} |u_n|_2^2 + \frac{C \cdot C_\varepsilon}{2} D^{\frac{(p-2)N}{\alpha}} \cdot |u_n|_2^{p - \frac{(p-2)N}{\alpha}}.
$$
Denote $\tilde{C}_\varepsilon = \frac{C \cdot C_\varepsilon}{2} D^{\frac{(p-2)N}{\alpha}}$. Then
$$
\frac{c}{2} \leq \frac{\varepsilon}{2} |u_n|_2^2 + \tilde{C}_\varepsilon |u_n|_2^{p - \frac{(p-2)N}{\alpha}}.
$$
Take any $\varepsilon \leq \frac{c}{2 \left( \sup_n \| u_n \| \right)^2}$. Then
$$
\frac{c}{2} \leq \frac{c |u_n|_2^2}{4 \left( \sup_n \| u_n \| \right)^2}  + \tilde{C}_\varepsilon |u_n|_2^{p - \frac{(p-2)N}{\alpha}}.
$$
Of course
$$
\frac{|u_n|_2^2}{\left( \sup_n \| u_n \| \right)^2} \leq 1
$$
and therefore
$$
\frac{c}{2} \leq \frac{c}{4}  + \tilde{C}_\varepsilon |u_n|_2^{p - \frac{(p-2)N}{\alpha}}.
$$
Finally
$$
\frac{c}{4 } \leq \tilde{C}_\varepsilon |u_n|_2^{p - \frac{(p-2)N}{\alpha}}.
$$
Therefore
$$
\ln \frac{c}{4} \leq \left( p - \frac{(p-2)N}{\alpha} \right) \ln \left( C_1 |u_n|_2 \right),
$$
where $C_1 = \tilde{C}_\varepsilon^{\frac{1}{p - \frac{(p-2)N}{\alpha}}}$. Thus
$$
\frac{\alpha}{\alpha p - (p-2)N} \ln \frac{c}{4} \leq \ln (C_1 |u_n|_2 ).
$$
And finally
$$
|u_n|_2^2 \geq \left( \frac{1}{C_1} \exp \left( \frac{\alpha}{\alpha p - (p-2)N} \ln \frac{c}{4} \right) \right)^2 =: \tilde{c} > 0.
$$
Take any $R > 0$ and observe that
$$
|u_n|_2^2 = \int_{B(0,R)} |u_n|^2 \, dx + \int_{\mathbb{R}^N \setminus B(0,R)} |u_n|^2 \, dx.
$$
Assume by a contradiction that $u_0 = 0$, i.e. $u_n \to 0$ in $L^2_{loc} (\R^N)$. Then for every $R > 0$ there is $n_0$ such that for $n \geq n_0$ we have
$$
\int_{B(0,R)} |u_n|^2 \, dx \leq \frac{\tilde{c}}{2}.
$$
Then 
$$
\int_{\mathbb{R}^N \setminus B(0,R)} |u_n|^2 \, dx \geq \frac{\tilde{c}}{2}.
$$
On the other hand
\begin{eqnarray*}
\frac{\tilde{c}}{2} &\leq& \int_{\mathbb{R}^N \setminus B(0,R)} |u_n|^2 \, dx = \int_{\mathbb{R}^N \setminus B(0,R)} \frac{V(x) |u_n|^2}{V(x)} \, dx \leq \frac{1}{\inf_{|x|\geq R} V(x)} \int_{\R^N \setminus B(0,R)} V(x) |u_n|^2 \, dx \\ &\leq& \frac{\|u_n\|^2}{\inf_{|x|\geq R} V(x)}  \leq \frac{\sup_n \|u_n\|^2}{\inf_{|x| \geq R} V(x)}.
\end{eqnarray*}
Taking $R > 0$ big enough we obtain a contradiction, since $V(x) \to \infty$ as $|x| \to \infty$. Therefore $u_0 \neq 0$, which completes the proof.

\end{altproof}

\parbox{9cm}{
\noindent \\{\sc Address of the author:}\\
Bartosz Bieganowski, bartoszb@mat.umk.pl,\\
 Nicolaus Copernicus University \\
 Faculty of Mathematics and Computer Science\\
 ul. Chopina 12/18, 87-100 Toru\'n, Poland\\
 }


\begin{thebibliography}{99}
\baselineskip 2 mm






\bibitem{AlamaLi} S. Alama, Y. Y. Li: {\em On ''multibump'' bound states for certain semilinear elliptic equations},  Indiana Univ. Math. J. {\bf 41}, (1992), no. 4, 983--1026. 

\bibitem{Ambrosio} V. Ambrosio: {\em Ground states for superlinear fractional Schr\"{o}dinger equations in $\mathbb{R}^{N}$}, arXiv:1601.06284.








\bibitem{BartschMederski1} T. Bartsch, J. Mederski: {\em Ground and bound state solutions of semilinear time-harmonic Maxwell equations in a bounded domain}, Arch. Rational Mech. Anal. {\bf 215} (1), (2015), 283--306.

\bibitem{BerestyckiLions} H. Berestycki, P.-L. Lions: {\em Nonlinear scalar field equations. I. Existence of a ground state}, Arch. Rational Mech. Anal. {\bf 82} (1983), no. 4, 313--345.

\bibitem{BieganowskiMederski} B. Bieganowski, J. Mederski: {\em Nonlinear Schr\"odinger equations with sum of periodic and vanishig potentials and sign-changning nonlinearities}, arXiv:1602.05078.










\bibitem{Binlin} Z. Binlin, M. Squassina, Z. Xia: {\em Fractional NLS equations with magnetic field, critical frequency and critical growth}, arXiv:1606.08471.


\bibitem{BuffoniJeanStuart} B. Buffoni, L. Jeanjean, C. A. Stuart: {\em Existence of a nontrivial solution to a strongly indefinite semilinear equation}, Proc. Amer. Math. Soc. {\bf 119} (1993), no. 1, 179--186. 









\bibitem{Cabre} X. Cabr\'{e}, Y. Sire: {\em Nonlinear equations for fractional Laplacians I: regularity, maximum principles, and Hamiltonian estimates}, Ann. Inst. H. Poincare Anal. Non Lineaire {\bf 31}, (2014), 23--53.

\bibitem{Chen} G. Chen, Y. Zheng: {\em Fractional nonlinear Schr\"{o}dinger equations with singular potential in $\mathbb{R}^n$}, arXiv:1511.09124.

\bibitem{Cheng} M. Cheng: {\em Bound state for the fractional Schr\"odinger equation with unbounded potential}, J. Math. Phys. {\bf 53}, 043507 (2012).

\bibitem{CotiZelati} V. Coti-Zelati, P. Rabinowitz: {\em Homoclinic type solutions for a semilinear elliptic PDE on $\R^n$}, Comm. Pure Appl. Math. {\bf 45}, (1992), no. 10, 1217--1269.

\bibitem{Davila2} J. D\'avila, M. del Pino, S. Dipierro, E. Valdinoci: {\em Concentration phenomena for the nonlocal Schr\"odinger equation with Dirichlet datum}, Anal. PDE {\bf 8} (2015), no. 5, 1165--1235. 

\bibitem{Davila} J. D\'avila, M. del Pino, J. Wei: {\em Concentrating standing waves for the fractional nonlinear Schr\"odinger equation}, J. Differential Equations {\bf 256} (2014), no. 2, 858--892.

\bibitem{DiNezza} E. Di Nezza, G. Palatucci, E. Valdinoci: {\em Hitchhiker's guide to the fractional Sobolev spaces}, Bulletin des Sciences Mathematiques {\bf 136} (2012), no. 5, 521--573.




\bibitem{Dipierro} S. Dipierro, G. Palatucci, E. Valdinoci: {\em Existence and symmetry results for a Schr\"odinger type problem involving the fractional Laplacian}, Le Matematiche {\bf 68}, (2013), no. 1.

\bibitem{doO} J.M. do \'O, O.H. Miyagaki, M. Squassina: {\em Critical and subcritical fractional problems with vanishing potentials}, Commun. in Contemp. Math. {\bf 18}, 1550063 (2016), DOI:http://dx.doi.org/10.1142/S0219199715500637.




\bibitem{Fall2} M.M. Fall, F. Mahmoudi, E. Valdinoci: {\em Ground states and concentration phenomena for the fractional Schr\"odinger equation}, Nonlinearity {\bf 28} (2015), no. 6, 1937--1961.

\bibitem{Fall} M.M. Fall, E. Valdinoci: {\em Uniqueness and Nondegeneracy of Positive Solutions of $(-\Delta)^s u + u = u^p$ in $\mathbb{R}^N$ when $s$ is close to $1$}, Comm. Math. Phys. {\bf 329} (2014), no. 1, 383--404.

\bibitem{Frank} R.L. Frank, E. Lenzmann: {\em Uniqueness of non-linear ground states for fractional Laplacians in $\mathbb{R}$}, Acta Math. {\bf 210}, no. 2,  (2013), 261--318. 

\bibitem{Frank2} R.L. Frank, E. Lenzmann, L. Silvestre: {\em Uniqueness of Radial Solutions for the Fractional Laplacian}, Comm. Pure Appl. Math., Vol. 69, Issue 9 (2016), 1671--1726.

\bibitem{Guedes} I. Guedes: {\em Solution of the Schr\"odinger equation for the time-dependent linear potential}, Phys. Rev. A {\bf 63}, (2001), 034102.

\bibitem{GuoXu} X. Guo, M. Xu: {\em Some physical applications of fractional Schr\"{o}dinger equation}, J. Math. Phys. {\bf 47}, (2006), 082104 .

\bibitem{He} X. He, M. Squassina, W. Zou: {\em The Nehari manifold for fractional systems involving critical nonlinearities}, Commun. Pure Applied Anal. {\bf 15}, (2016), 1285--1308.





\bibitem{KryszSzulkin} W. Kryszewski, A. Szulkin, {\em Generalized linking theorem with an application to semilinear Schr\"odinger equation}, Adv. Diff. Eq. {\bf 3}, (1998), 441--472.

\bibitem{Laskin2000} N. Laskin: {\em Fractional quantum mechanics and L\'evy path integrals}, Phys. Lett. A {\bf 268}, 298--305 (2000).

\bibitem{Laskin2002} N. Laskin: {\em Fractional Schr\"odinger equation}, Phys. Rev. E {\bf 66}, 056108 (2002).




\bibitem{LiemertKienle} A. Liemert, A. Kienle: {\em Fractional Schr\"odinger Equation in the Presence of the Linear Potential}, Mathematics {\bf 4}, (2016), 31. 







\bibitem{Longhi} S. Longhi: {\em Fractional Schr\"{o}dinger equation in optics.}, Opt. Lett. {\bf 40} (2015), 1117--1120.




\bibitem{MederskiTMNA2014} J. Mederski: {\em Solutions to a nonlinear Schr\"odinger equation with periodic potential and zero on the boundary of the spectrum}, Topol. Methods Nonlinear Anal. {\bf 46}, (2015), 755--771.









\bibitem{Rabinowitz:1992} P.H. Rabinowitz: {\em On a class of nonlinear Schr\"odinger equations}, Z. Angew. Math. Phys. {\bf 43}, (1992), 270--291.


\bibitem{Robinett} R.W. Robinett: {\em Quantum mechanical time-development operator for the uniformly accelerated particle}, Am. J. Phys. {\bf 64}, (1996), 803--807.

\bibitem{Ros} X. Ros-Oton, J. Serra: {\em The Dirichlet problem for the fractional Laplacian: regularity up to the boundary}, J. Math. Pures Appl. {\bf 101} (2014), 275--302.

\bibitem{secchi} S. Secchi: {\em Ground state solutions for nonlinear fractional Schr\"{o}dinger equations in $\mathbb{R}^N$}, Journal of Mathematical Physics {\bf 54}, 031501 (2013); doi: 10.1063/1.4793990.

\bibitem{secchiTMNA} S. Secchi: {\em On fractional Schr\"odinger equations in $\mathbb{R}^N$ without the Ambrosetti-Rabinowitz condition}, Topol. Methods Nonlinear Anal. {\bf 47}, (2016), 19--41.




\bibitem{Struwe} M. Struwe: {\em Variational Methods}, Springer 2008.


\bibitem{SzulkinWeth} A. Szulkin, T. Weth: {\em Ground state solutions for some indefinite variational problems}, J. Funct. Anal. {\bf 257}, (2009), no. 12, 3802--3822. 

\bibitem{Torres} C.E. Torres Ledesma: {\em Non-homogeneous fractional Schr\"{o}dinger equation}, Journal of Fractional Calculus and Applications, Vol. 6(2) July 2015, pp. 108--114.









\bibitem{ZhangLiu} Y. Zhang, X. Liu, M.R. Beli\'c, W. Zhong, Y. Zhang, M. Xiao: {\em Propagation Dynamics of a Light Beam in a Fractional Schr\"{o}dinger Equation}, Phys. Rev. Lett. {\bf 115}, (2015), 180403.

\bibitem{ZhangZhong} Y. Zhang, H. Zhong, M.R. Beli\'c, N. Ahmed, Y. Zhang, M. Xiao: {\em Diffraction-free beams in fractional Schr\"{o}dinger equation}, arXiv:1512.08671.

\end{thebibliography}
\end{document}